\index{}

\documentclass[12pt,reqno]{amsart}
\usepackage{amssymb , amsmath, amsthm}
\usepackage{amsmath}
\usepackage{amsbsy}
\usepackage{amssymb}
\usepackage{color}
\usepackage[dvips]{graphicx}
 \usepackage{subfigure}
\usepackage[active]{srcltx}
\usepackage[active]{srcltx}

\usepackage{amsmath}
\usepackage{amsfonts}
\usepackage{amssymb}
\usepackage{amstext}
\usepackage{amsbsy}
\usepackage{epsf}

\newtheorem{theo}{Theorem}[section]
\newtheorem{prop}{Proposition}[section]
\newtheorem{coro}{Corollary}[section]
\newtheorem{lemma}{Lemma}[section]
\theoremstyle{definition}
\newtheorem{definition}{Definition}[section]

\newcommand{\defref}[1]{Definition~\ref{#1}}

\def\r{\mathbb R}

\def\h{\mathbb H}

\let\hip=\Hip

\newcommand{\R}{\mathbb R}

\newcommand{\C}{\mathbb C}

\newcommand{\N}{\mathbb N}
\newcommand{\Ne}{{\mathbb N}^*}

\newcommand{\wt}{\widetilde}

\renewcommand{\Re}{{\rm Re}}
\renewcommand{\Im}{{\rm Im}}

\let\h=\hip

\let\re=\R

\def \ga{\gamma}

\def \ga{\gamma}

\def\rmd{\mathop{\rm d\kern -1pt}\nolimits}
\def\rme{\mathop{\rm e\kern -1pt}\nolimits}

\def\bel{ \medskip
 \centerline{$ \ast \hbox to 1.0cm{}\ast \hbox to 1.0cm{}\ast $}
}

\def\grad{\nabla}

\def\longerrightarrow{-\kern-5pt\longrightarrow}

\def\star{\lower 1pt\hbox{*}}
\def \nulset {
\raise 1pt\hbox{ \hskip -3pt$\not$\kern -0.2pt \raise
.7pt\hbox{${\scriptstyle\bigcirc}$}}}

\def\grad{\nabla}

\newcommand{\hd}{\mathbb{H}^2}
\newcommand{\sd}{\mathbb{S}^2}
\newcommand{\hi}[1]{\mathbb{H}^#1}

\newcommand{\ch}{\cosh}
\newcommand{\sh}{\sinh}

\newcommand{\pain}{\partial_{\infty}}

\newcommand{\ov}[1]{\overline{#1}}

\let\leq=\leqslant
\let\geq=\geqslant

\newtheorem{step}{Step}

\begin{document}

\title[minimal stable vertical planar end]{A
minimal stable vertical planar end in $\hi2\times \r$ \\ has
finite total curvature}

\author[R. Sa
Earp $\ $ and $\ $ E. Toubiana  ]{
 Ricardo Sa Earp and
Eric Toubiana}

 \address{Departamento de Matem\'atica \newline
  Pontif\'\i cia Universidade Cat\'olica do Rio de Janeiro\newline
Rio de Janeiro \newline
22453-900 RJ \newline
 Brazil }
\email{earp@mat.puc-rio.br}

\address{Institut de Math\'ematiques de Jussieu - Paris Rive Gauche \newline
Universit\'e Paris Diderot - Paris 7 \newline
Equipe G\'eom\'etrie et Dynamique,  UMR 7586 \newline
B\^atiment Sophie Germain \newline
Case 7012 \newline
75205 Paris Cedex 13 \newline
France}
\email{toubiana@math.jussieu.fr}

\thanks{
\\
The authors were partially supported by CNPq
and FAPERJ of Brasil.}

\date{\today}

\begin{abstract}
We prove that a minimal oriented stable  annular end in
$\hi2\times \r$
whose asymptotic boundary is contained in two vertical lines
 has finite total curvature and converges to a vertical plane. Furthermore, if
the end is  embedded then
 it is a horizontal graph.
\end{abstract}


\maketitle

\section{Introduction}

Since the last decades there is an increasing  interest among geometers to study
minimal surfaces in $\hi2\times\r$ with a certain prescribed asymptotic
boundary, where $\hi2$ stands for the hyperbolic plane.
\smallskip

In a joint work with B. Nelli  \cite{NST}, the authors
characterized the
catenoids
in $\hi2\times\r$
among  minimal  surfaces with the same asymptotic boundary. That is, a
connected and complete minimal surface whose asymptotic boundary is the union
of two distinct copies of $\pain \hi{2}$ is a catenoid. Noticing that the
catenoids are the unique minimal surfaces of revolution in $\hi2 \times \r$
and each catenoid  $M$ has infinite total curvature, that is
$\int_M |K|\, dA=\infty$, where $K$ is the Gaussian curvature of $M$.

\smallskip

In this paper, we prove that a minimal oriented stable  annular end in
$\hi2\times \r$
whose asymptotic boundary is contained in two vertical lines
 has
finite total curvature and converges to a vertical plane (Theorem
\ref{Main Theorem}). Furthermore, if the end is embedded then it is a
horizontal graph with respect to a geodesic in $\hi2\times \{0\}$
(\defref{D.horizontal graph}).

\smallskip

We point out that in Euclidean space a famous result
of D. Fisher-Colbrie \cite{F-C}
states that
a complete oriented minimal surface has finite index if and only if it has
finite total curvature.
Observe that in $\hi2 \times \r$, finite total curvature of a complete oriented
minimal surface implies finite index
\cite{B-SE}, but the converse does not hold: there are many examples of oriented complete
stable minimal surfaces with infinite total curvature.

\smallskip

Indeed,
 there are families of oriented complete stable minimal
surfaces
invariant by a nontrivial group of screw-motions \cite{Sa}.
A particular example is a connected, complete and stable minimal surface
whose
asymptotic boundary is the union of an arc in $\pain \hi2 \times \{0\}$ with
the two upper half vertical lines issuing from the two boundary points
$\{p_\infty, q_\infty\}$ of the
arc \cite[Proposition 2.1-(2)]{SE-T}. Furthermore, this surface has strictly
negative Gaussian curvature and
is invariant by
any translation along the geodesic line whose asymptotic boundary is
$\{p_\infty, q_\infty\}$, and so does not has  finite total curvature.

\smallskip

A second example is given by  a  one parameter family of entire horizontal
graphs with respect to a geodesic $\ga$ of $\hi 2 \times \{0\}$, that is,  each
one is a horizontal graph (see Definition \ref{D.horizontal graph}) over an
entire vertical plane orthogonal to $\gamma.$
In fact, each graph is stable and is invariant by hyperbolic screw motions,
hence  it has infinite total curvature.
Moreover, it is also an entire vertical graph,
 given by a simple explicit formula taking the half-plane model for $\hi 2$
\cite[equation (4), section 4]{Sa} or \cite[example 2, section 3]{SE-T}.
For each surface, the intersection with any slice is a geodesic
of $\hi 2$, which depends on the slice. Therefore, the asymptotic boundary of
each surface is constituted of two analytic entire symmetric curves,
each curve is a "exp type" graph
over a vertical line (Figure \ref{Fig-Intro}). We deduce that the asymptotic
boundary is not contained in two vertical lines.
 Thus, each surface is a minimaly embedded plane in $\hi 2 \times \R$, which is
stable and has infinite total curvature. This shows that the hypothesis about
the asymptotic boundary in Theorem \ref{Main Theorem} cannot be removed.

\begin{figure}[!h]
\centerline{\includegraphics[scale=0.25]{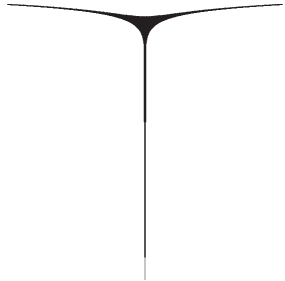}}
\caption{\em Entire minimal horizontal graph with infinite total
curvature  \newpage $  $
  \hskip40mm
$($view from the outside of $\hi2 \times \r)$}
\label{Fig-Intro}
\end{figure}

Another  example is given by an end of a catenoid.
Of course, a slice
$\hi2 \times\{t\}$ is a trivial example.

\smallskip

We observe that there is  another notion of horizontal graph
in $\hi 2 \times \R$ that appears in the literature \cite{Sa}.

\smallskip

We point out that as an immediate  consequence  of our main theorem, we get
an extension (Corollary \ref{C.model}) of the Schoen
type theorem proved by
L. Hauswirth, B. Nelli and the authors \cite{HNST}. We
remark that to
accomplish this task, we
use some results established in \cite{HNST}.  A crucial property  shown in
\cite{HNST}  is that
the horizontal sections of finite total curvature ends converge towards a
horizontal geodesic of $\hi 2\times \r$.

\subsection*{Acknowledgements}
{\Small The first author
wishes to thank {\em Laboratoire} {\em G\'eom\'etrie et Dynamique de l'Institut
de
Math\'ematiques de Jussieu} for the kind hospitality.  The second author wishes
to thank
{\em Departamento de Matem\'atica da
PUC-Rio} for the kind hospitality.}

\section{main theorem}\label{Sec.End}

First, we need to fix some definitions and terminologies.

We choose the Poincar\'e disk model for the hyperbolic plane $\hi2$.

We identify $\hi2$ with $\hi2 \times \{0\}$. We define a
{\em vertical plane} in $\hi2 \times \r$ to be a product $\gamma \times \r$, where
$\gamma\subset \hi2$ is a complete geodesic line.

\begin{definition}[Vertical planar end]
 We say that $L\subset \partial_{\infty}(\h^2\times\r)$ is a
 {\em vertical line}
if $L=\{p_\infty\}\times\r$ for some $p_\infty\in\partial_{\infty}\h^2$.

We say that a complete  minimal surface $E$ immersed in
$\h^2\times\r$ with
compact boundary is {\em a vertical planar end},  if the
surface is
an oriented properly immersed annulus whose  asymptotic boundary is
contained
in two distinct
vertical lines $L_1$ and $L_2$. Precisely, there is a vertical plane
$P\subset \hi2 \times \r$ such that
$\pain E \cap (\pain \hi2 \times \r )\subset \pain P$.
 \end{definition}

\begin{definition}
 Let $S\subset \h^2\times\r$ be a surface and let
 $P=\gamma \times \r$ be a vertical plane.

 For  any positive real number $\rho$,
we denote by
$L_\rho^+,L_\rho^-\subset\hi2$ the
two equidistant lines of $\gamma$ at distance $\rho$. Let
$Z_\rho$ be the component of
$(\hi2\times \r)\setminus(L_\rho^+ \cup L_\rho^-)\times\r$
containing $P$.

We say that $S$ {\em converges
to the vertical plane $P$} if the following two properties hold:
\begin{enumerate}
\item For any $\rho >0$ there is a compact part $K_\rho$ of
$S$  such that
$ S\setminus K_\rho \subset Z_\rho$.

\item $\pain S =\pain P$.
\end{enumerate}
\smallskip

 We observe that this definition and the definition of {\em asymptotic to the vertical plane
$P$} given in \cite[Definition 3.1]{HNST} are distinct. But, under geometric
assumptions they lead to the same conclusion. Compare Corollary \ref{C.model}
below with \cite[Theorem 3.1]{HNST}.
\end{definition}

We  recall now the definition of {\em horizontal graph with respect to
a geodesic} given in \cite[Definition 3.2]{HNST}.

\begin{definition}\label{D.horizontal graph}
Let $\gamma \subset \hi2$ be a geodesic. We say that a nonempty set
$S \subset \hi2 \times \R$ is a {\em horizontal graph with respect to
the geodesic
$\gamma$}, or simply a {\em horizontal graph}, if for any equidistant line
$\wt \gamma$ of $\gamma$ and for any $t \in \R$, the curve $\wt \gamma \times
\{t\}$
intersects $S$ at most in one point.\end{definition}
\bigskip

We state now precisely our main theorem:

\begin{theo}\label{Main Theorem}
{\em  Let $E$ be a minimal stable vertical planar end in
$\hi2\times\r$ and let $P$ be the vertical plane such that
$\pain E \cap (\pain \hi2 \times \r)\subset \pain P$.

Then, $E$  has finite total curvature and converges to the vertical plane $P$.
Furthermore if $E$ is embedded then,
up to a compact part, $E$ is a horizontal graph.  }
\end{theo}

\bigskip

There are many examples
of complete, possibly with compact boundary,
minimal surfaces whose asymptotic boundary is contained in the union of
vertical lines or copies of the asymptotic boundary of $\hi2$.
For instance, we refer to the first paper on this subject written by
B. Nelli and H. Rosenberg \cite{NR}. We remark that the first paper about
minimal ends of finite total curvature in $\hi2 \times \r$ was carried
out by L. Hauswirth and H. Rosenberg \cite{HR}.

\smallskip

We recall now that F. Morabito and M. Rodriguez \cite{MR} and J. Pyo \cite{P}
have constructed,
independently, a family of minimal embedded annuli with
finite total curvature. Each end of such annuli is  asymptotic to  a vertical
geodesic plane. In \cite{HNST} we  called  each of such  surfaces   a {\em two
ends model surface}.  The following corollary extends the main theorem of
\cite{HNST}.

\begin{coro}\label{C.model}
 An oriented complete and connected minimal surface immersed in $\hi2\times\r$
with two
distinct embedded annular ends, each one being stable and the asymptotic
boundary of
each end being contained in the asymptotic boundary of a vertical geodesic
plane, is a  two ends model surface.
\end{coro}

\bigskip

\noindent {\em Proof of Theorem \ref{Main Theorem}.}

We have  $P=\gamma\times\r$, where $\gamma$ is a geodesic of $\hi2$.
We set $\pain \gamma=\{p_\infty, q_\infty\}$.

Given any isometry $T$ of
$\hi2$ we denote also by $T$ the isometry of
$\hi2 \times \r$ induced by $T$: $(x,t)\mapsto (T(x),t)$.

For any geodesic $\alpha\subset\hi2$ we set $P_\alpha:=\alpha \times\r$, that
is
$P_\alpha$ is the vertical plane containing $\alpha$.

\bigskip

We will proceed the proof of Theorem \ref{Main Theorem} in several steps.

\smallskip
\step \label{halfspace} {\em Let $\alpha\subset\hi2$ be any geodesic such
that
$\alpha \cap \gamma= \emptyset$. If a component of
$(\hi2\times\r)\setminus P_\alpha$, say $P_\alpha^+$,
contains $P\cup \partial E$ then $E\subset P_\alpha^+$.}

\smallskip

This is a consequence of the maximum principle established by the authors and
B. Nelli  \cite[Theorem 3.1]{NST}.

\bigskip

\step \label{asymptotic} {\em
For any $\rho >0$ there is a compact part $K_\rho$ of
$E$  such that
$ E\setminus K_\rho \subset Z_\rho$. Thus,
$\pain E \subset \pain P$. }

\medskip

\noindent {\em Proof of Step \ref{asymptotic}.} Let $\rho$ be a positive
real number. Assume that $E\setminus Z_\rho$
is not compact,
then there is an unbounded sequence $p_n=(x_n,t_n) \in  E\setminus Z_\rho$.
We can assume that the sequence $(p_n)$ belongs to the component
of $(\hi2 \times \r)\setminus (L_\rho^+ \times \r)$ which does not contain $P$.

\begin{figure}[!h]
\hspace{-8em}
\includegraphics[scale=0.9]{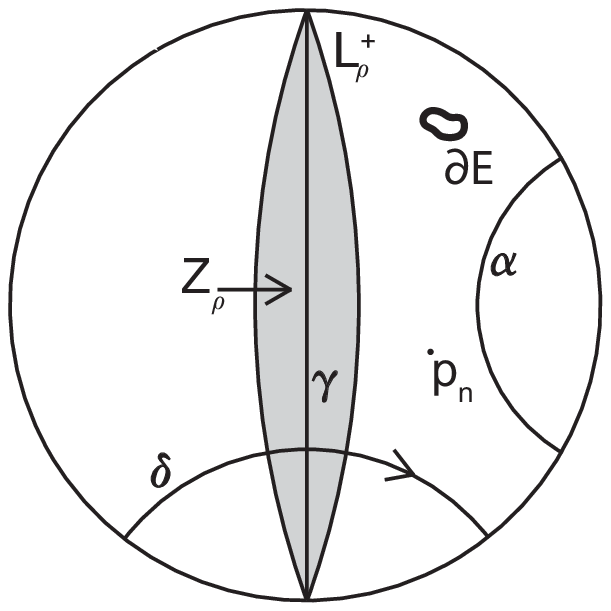}
\caption{}
\label{Fig-Step 2}
\end{figure}

Let
$\alpha \subset \hi2$ be any geodesic such that
$ \alpha \cap  \gamma = \emptyset$ and
such that $\partial E$ and $P$ belong to the same component of
 $(\hi2 \times \r) \setminus P_\alpha$ (Figure \ref{Fig-Step 2}). We denote by
$P_\alpha^-$ the other
component of $(\hi2 \times \r) \setminus P_\alpha$.
We deduce from Step \ref{halfspace} that
 $P_\alpha^-\cap E=\emptyset$. This implies that no subsequence  of $(x_n)$
 converges  either to $p_\infty$ or to $q_\infty.$ Indeed, if a subsequence of
$(x_n)$ would converge to $p_\infty,$ we
could choose the  geodesic $\alpha$ such that $p_\infty\in \pain \alpha,$
contradicting  $P_\alpha^-\cap E=\emptyset.$ The same argument shows that there
is no subsequence of $(x_n)$ that converges
to other points of $\pain \hi 2.$

Then, we deduce from above and from  the assumptions about the asymptotic boundary of
$E$ that
   $(x_n)$ is a
bounded sequence in $\hi2$. Thus, up to extract a subsequence of $(p_n)$ and up to  a
reflection, we may assume that $t_n \to +\infty$.

Now we consider the family of complete minimal surfaces
$M_d$, $d>1$,
described
in \cite[Proposition 2.1-(1)]{SE-T} and in the
proof of \cite[Theorem 3.1]{NST}. Recall that $M_d$ contains the equidistant
line $L_\rho^+ $
staying at the distance $\rho= \cosh^{-1}(d)$ from $\gamma$
and  that $M_d$ is contained in the closure of the non mean
convex component
of   $(\hi2 \times \r)\setminus (L_\rho^+ \times \r)$.
By abuse of notation,
this surface is denoted
$M_\rho$. The proof of the assertion follows from the maximum
principle again \cite[Theorem 3.1]{NST}, using
the surfaces $M_\rho$. We give a short proof in the sequel for the readers
convenience.

Let $V_\rho$ be the closure of the component of
$(\hi2\times \r) \setminus M_\rho$ which does not contain $P$. Observe that the
height function is bounded on $V_\rho$. From the considerations above, there
exists $t>0$ such that
\begin{itemize}
 \item $\partial E \cap \big(V_\rho + (0,0,t)\big) =\emptyset$,
 \item $E \cap \big(V_\rho + (0,0,t)\big) \not=\emptyset$.
\end{itemize}

Let $\delta\subset \hi2$ be a geodesic orthogonal to $\gamma$. Using that $E$
is properly
immersed, moving horizontally $M_\rho +(0,0,t)$  along $\delta$ and considering
horizontal translated copies of $M_\rho +(0,0,t),$  we must find  a
last contact at an interior point of $E$ and a copy of  $M_\rho +(0,0,t)$. This
yields a
contradiction with the maximum principle. \qed

\bigskip

Let $p_0 \in \hi2 \times \r$ be a fixed point. For any $r >0$ we denote by
$B_{r}$ the open geodesic ball in $\hi2 \times \r$ centered at $p_0$ with
radius $r$.

\step \label{vertical} {\em Let $n_3$ be  the third coordinate of the
unit normal field on $E$ with respect to the product metric
on $\hi2 \times \r$.
We have that $n_3 (p) \to 0$ uniformly when $p \to
\pain E$.

More
precisely, for any $\varepsilon >0$, there exists $\rho>0$  such that
for any
$p \in (E\cap Z_\rho)\setminus B_{1/\rho}$, we have
$|n_3(p)|<\varepsilon$. }

\medskip

\noindent {\em Proof of Step \ref{vertical}.}
Assume by contradiction that the assertion does not hold.  Therefore
there exist $\varepsilon >0$ and a sequence of points
$p_n:=(x_n,t_n) \in E\cap Z_{1/n}$, $n\in \Ne$, such that
$| n_3(p_n)|>\varepsilon$ and the sequence $(p_n)$ is not bounded in
$\hi2\times \r$.

 Up to extract a subsequence, we can assume that the sequence
 $(p_n)$ converges in $(\hi2\times \r) \cup \pain (\hi2\times \r)$, eventually
with $\lim x_n\in\pain\hi2$ and $\lim t_n =+ \infty$ (up to a
vertical reflection).

\medskip

\noindent {\bf First case:} Suppose that $\lim x_n\in\pain\hi2$, thus we have
$\lim x_n\in\pain\gamma$. Without loss of generality, we can assume that
$x_n\to p_\infty$.

We consider on the geodesic $\gamma$ the orientation given by
$p_\infty \to q_\infty$.
We choose two points $y_1, y_2 \in \gamma$ such that $y_1< y_2$. Let
$C_i\subset \hi2$ be the geodesic orthogonal to $\gamma$ through $y_i$,
$i=1,2$, (Figure \ref{Fig-Step 3a}).
\begin{figure}[!h]
 \vspace{-15mm}
\centerline{
\subfigure[\hspace{-6em}]
{
\includegraphics[scale=0.9]{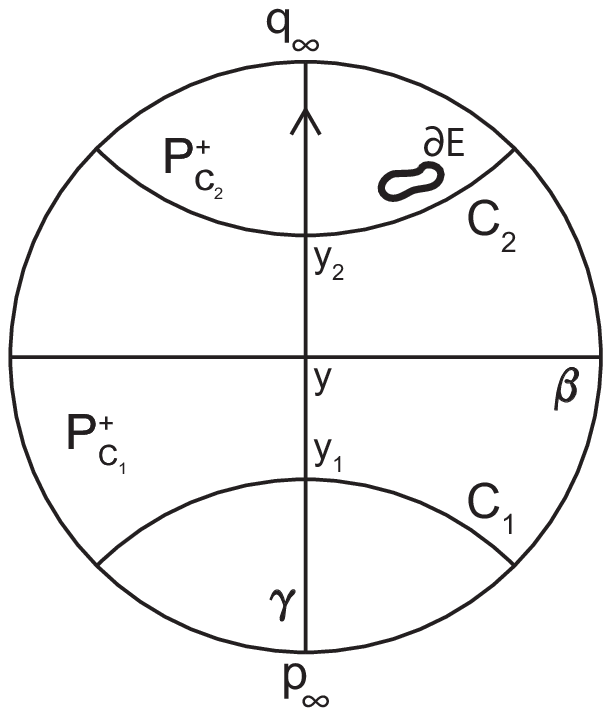}
\label{Fig-Step 3a}
} \hspace{-10.em}
\subfigure[\hspace{-5em}]{
\includegraphics[scale=0.9]{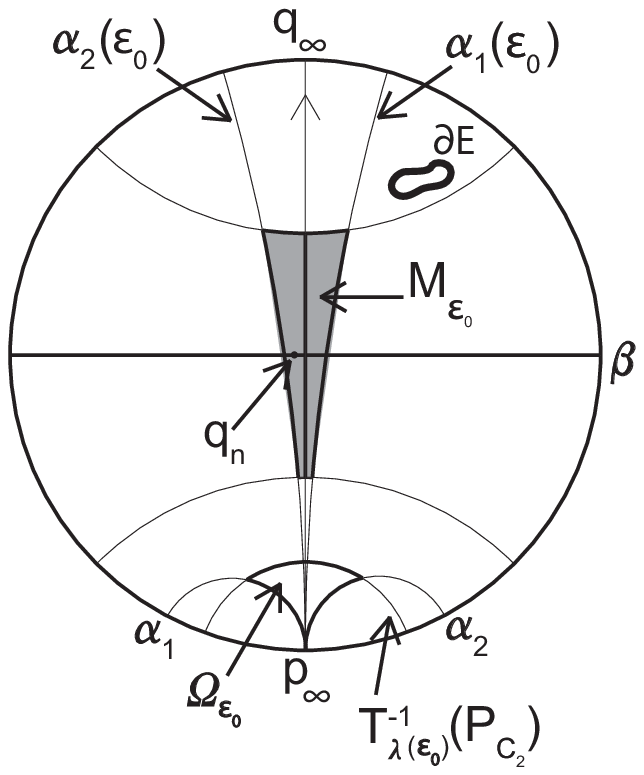}
\label{Fig-Step 3b}
}
}
\caption{}
\end{figure}

Let $P_{C_i}^ +$ be the connected component
of $(\hi2 \times\r)\setminus P_{C_i}$ containing $q_\infty$ in its
asymptotic boundary, we
denote by  $P_{C_i}^ -$ the other component. Since $E$ has compact boundary,
we can choose $y_1$ et $y_2$ so that
$\partial E\subset P_{C_2}^ +$. Let $y\in \gamma$ be the midpoint of the
geodesic segment $[y_1,y_2]$ of $\gamma$ and let
$\beta\subset \hi2$ be the geodesic orthogonal to $\gamma$ through $y$.
Observe that the  vertical planes $P_{C_1}$ and
$P_{C_2}$ are symmetric with respect to $P_\beta$.

\medskip

Let $\alpha_1\subset \hi2$ be a geodesic so that
$p_\infty\in \pain \alpha_1$, $\pain \alpha_1 \cap \pain C_1=\emptyset$ and
$\alpha_1 \cap  C_1=\emptyset$, thus
$\alpha_1 \subset P_{C_1}^ - $, (Figure \ref{Fig-Step 3b}).

We denote by $\alpha_2\subset \hi2$ the symmetric of $\alpha_1$ with
respect to $\gamma$. Using the Step \ref{halfspace}, we can deduce
 that the component of
 $(\hi2 \times\r)\setminus (P_{\alpha_1}\cup P_{\alpha_2})$
containing $P$ also contains $E \cup \partial E$.

\medskip

Let $\varepsilon_0>0$. For any $\lambda>0$ we denote by $T_\lambda$ the
hyperbolic translation of length $\lambda$
along $\gamma$, with the orientation $p_\infty \to q_\infty$.
 There exists $\lambda (\varepsilon_0)>0$ so that the
geodesics $T_{\lambda (\varepsilon_0)}(\alpha_1)$
and $T_{\lambda (\varepsilon_0)}(\alpha_2)$ belong to a
$\varepsilon_0$-neighborhood of
 $\gamma$ in $\hi2 \cup \pain\hi2$, in the Euclidean meaning.
We set
$T_{\lambda (\varepsilon_0)}(\alpha_i):=
 \alpha_i(\varepsilon_0)$, $i=1,2$. We remark that
the vertical planes $P_{\alpha_i(\varepsilon_0)}$ belong to
 a $\varepsilon_0$-neighborhood  of $P_\gamma$ (in the Euclidean meaning).

\medskip

Let $U_{\varepsilon_0}$ be the component of
$(\hi2\times\r)\setminus (P_{\alpha_1(\varepsilon_0)}\cup
P_{\alpha_2(\varepsilon_0)})$ containing $P$.
We denote by $M_{\varepsilon_0}$ the component of
$(\hi2\times\r)\setminus (P_{\alpha_1(\varepsilon_0)}\cup
P_{\alpha_2(\varepsilon_0)}
\cup P_{C_1} \cup  P_{C_2})\times \r$ containing   $\{y\}\times \r$.
We have therefore
\begin{equation*}
 M_{\varepsilon_0}= U_{\varepsilon_0}\cap  P_{C_1}^+ \cap  P_{C_2}^- .
\end{equation*}

\medskip

Let $\Omega_{\varepsilon_0}$ be the component of
$(\hi2\times\r)\setminus (T_{\lambda (\varepsilon_0)}^{-1}(P_{C_2}) \cup
P_{\alpha_1} \cup P_{\alpha_2} )$
such that $\Omega_{\varepsilon_0}\cap P\not=\emptyset$.
Then, by construction, for any
$\lambda \geq \lambda (\varepsilon_0)$ and for any
$p\in \Omega_{\varepsilon_0}$, we have
$T_\lambda (p) \in U_{\varepsilon_0}$.

\medskip

We may assume that
$x_n\in \Omega_{\varepsilon_0}$ for any $n$.

\medskip

For any $n\in \Ne$, there exists an unique $\lambda_n >0$
such that $T_{\lambda_n}(x_n)\in \beta$. Therefore,
setting $q_n:= T_{\lambda_n}(p_n)$, we have
$q_n\in P_\beta$.

For $n$ large enough, say $n>n_0 >0$, we have
$\lambda_n >\lambda (\varepsilon_0)$ (since $ x_n\rightarrow p_\infty$),
therefore
$q_n\in P_\beta \cap U_{\varepsilon_0}$, that is
$q_n\in M_{\varepsilon_0}$.

\medskip

For any $n>n_0$, we denote by $E_n (\varepsilon_0)$ the connected
component of
$T_{\lambda_n}(E) \cap M_{\varepsilon_0}$ containing $q_n$. By
construction,  $E_n (\varepsilon_0)$ is the component of
$T_{\lambda_n}(E\cap \Omega_{\varepsilon_0}) \cap M_{\varepsilon_0}$
containing $q_n$.
Consequently, the boundary of $E_n (\varepsilon_0)$ belongs to
$ P_{C_1} \cup  P_{C_2}$ and has no intersection with
$P_{\alpha_1(\varepsilon_0)}\cup
P_{\alpha_2(\varepsilon_0)}$:
\begin{equation}\label{Propriete-bord}
 \partial E_n (\varepsilon_0) \subset  P_{C_1} \cup  P_{C_2} \quad
\text{and}\quad
 \partial E_n (\varepsilon_0) \cap (P_{\alpha_1(\varepsilon_0)}\cup
P_{\alpha_2(\varepsilon_0)})=\emptyset.
\end{equation}

\medskip

Let $\ov {d_n}(\cdot,\cdot)$ be the intrinsic distance on
$T_{\lambda_n}(E)$.
By construction, for $n$ large enough, say
$n>n_1 >n_0$, we have
$\ov {d_n}\bigl(p,  \partial T_{\lambda_n}(E)\bigr) > \pi/2$
for any  $p\in E_n (\varepsilon_0)$.
Let $| \ov A_n |$ be the norm of the
second fundamental form of $T_{\lambda_n}(E)$.

Observe that the sectional curvature of $\hi2 \times \r$ is bounded
(in absolute value it is bounded by $1$).

Since the end $E$ is stable, the translated copy $T_{\lambda_n}(E)$ is also
stable for any $n>n_1$. Using the fact that the distances between $E_n
(\varepsilon_0)$ and the boundary of $T_{\lambda_n}(E)$
are uniformly bounded from below,
we deduce from
\cite[Main Theorem]{RST} that there exists a constant
$C^\prime >0$,
which does not depend on $n> n_1$ and neither on  $\varepsilon_0>0$, such that
\begin{equation}
 | \ov A_n (p)| < C^\prime
\end{equation}
for any  $p\in E_n (\varepsilon_0)$.

Furthermore, since the boundary of $E_n(\varepsilon_0)$ belongs to
$P_{C_1}\cup P_{C_2}$, there exists
a constant $C^{\prime  \prime }>0$, which does not depend on
$n$ and neither on $\varepsilon_0$, such that
\begin{equation*}
 \ov{d_n} (q_n, \partial E_n(\varepsilon_0))>C^{\prime  \prime }.
\end{equation*}
\smallskip
Now, we consider $\hi 2\times \r$ as an open set of Euclidean space $\r^3$, as well.
We deduce from  \cite[Proposition 2.3]{RST} and from Proposition
\ref{comparaison} in the Appendix, that there exists a real number
$\delta >0$, which does not depend on $n$ and neither on $\varepsilon_0$,
such that for any
$n>n_1$, a part $F_n$ of $E_n(\varepsilon_0)$ is the Euclidean graph
of a function defined on the
disk centered at point $q_n$ with Euclidean radius $\delta$ in the
tangent plane of  $E_n$ at $q_n$. Furthermore, the norm of the Euclidean gradient of this
function is bounded by $1$.

\smallskip

Let  $\nu$ be the unitary normal along  $E$ in the Euclidean metric.
We denote by  $\nu_3$ the vertical component of $\nu$. Recall that
$|n_3 (p_n)|>\varepsilon$, hence $|n_3 (q_n)|>\varepsilon$ for any $n$.
Comparing the product metric  of $\hi2 \times \r$ with the Euclidean metric,
it can be shown that there exists $\varepsilon^\prime >0$, which does not
depend on $n$, such that $|\nu_3 (q_n)|>\varepsilon^\prime$ for any $n>n_1$,
(see the formula of the unit normal vector field of a vertical graph in  the
proof of \cite[Proposition 3.2]{ST4}).
This implies that the tangent planes of $E_n$ at points $q_n$ have a slope
bounded below  uniformly (with respect to $n>n_1$).

Since the radius $\delta$ does not depend on $\varepsilon_0$, if we choose
$\varepsilon_0$  small enough, the Euclidean graph
$F_n$ will have nonempty intersection with
$P_{\alpha_1(\varepsilon_0)}$
and  $P_{\alpha_2(\varepsilon_0)}$, which is not possible.

\bigskip

\bigskip

\noindent {\bf Second case:} The sequence $(x_n)$ has a finite limit in $\hi2$.

Therefore, since the sequence $(p_n)$ is not bounded, up to considering a
subsequence, and up to a vertical reflection, we can assume that
$t_n \to +\infty$.
We saw in Step
\ref{asymptotic} that
for any $\rho>0$ there exists
 $t_\rho >0$ such that $E\cap \{ | t|>t_\rho\} \subset Z_\rho$.
 Then we can argue as in the first case replacing the vertical planes
$P_{\alpha_1(\varepsilon_0)}$ and
$P_{\alpha_2(\varepsilon_0)}$ by the surfaces
 $L_\rho^+\times \r$ and
$L_\rho^- \times \r$, recalling that
$L_\rho^+ \cup L_\rho^- =\partial Z_\rho$.

\medskip

Conjugating the two cases above we infer that
$n_3(p)\to 0,$
uniformly when
$p\to \pain E$. \qed

\vskip4mm

\noindent {\bf Notation.} The end is conformally parametrized  by
$U_R :=\{ z \in \C \mid 1\leq |z| < R\}$, for some $R>1$. The conformal,
complete and proper immersion
$X:U_R \rightarrow \hi2 \times \r$ is given by $X=(F,h)$ where
$F: U_R\rightarrow \hi2$ is a harmonic map and $h: U_R\rightarrow \r$ is a
harmonic function. Let $\sigma$ be the conformal factor of the hyperbolic
metric on $\hi2$, we set
$\phi :=(\sigma \circ F)^2 F_z \ov F _z  $. Since the immersion $X$ is
conformal,
we have $\phi = -h_z^2$, \cite[Proposition 1]{ST} and therefore, $\phi$ is
holomorphic.

 As above, we denote by  $n_3$  the third coordinate of the unit normal
field on $E$, with respect to the product metric. We define a
function
$\omega$ on $E$, or $U_R$, setting $n_3=\tanh \omega$. The induced metric on
$U_R$ is, \cite[Equation 14]{HST}:
\begin{equation}\label{metric}
 ds^2=4 \ch^2 (\omega)\, |\phi |\, |dz|^2.
\end{equation}

\bigskip

\step \label{end}
 {\em We have $R=\infty$, that is the end is conformally equivalent to a
punctured disk. Moreover $\phi$ extends meromorphically up to the end and
has the following expression on
$U  :=\{ z \in \C \mid 1\leq |z| \}$:
\begin{equation}\label{Formule-phi}
 \phi (z)=(\sum_{k\geq 1}\frac{a_{-k}}{z^k} +P(z))^2,
\end{equation}
where $P$ is a polynomial function.}

\medskip

\noindent {\em Proof of Step \ref{end}.}
From the expression of the metric $ds^2$, we deduce that
if $z_0\in U_R$ is a zero of $\phi$, then  $\omega$
must have a pole at $z_0$ and, therefore, $n_3 (z_0)=\pm 1$.
On the other hand, we infer from
Step \ref{vertical} that  $n_3(z)\to 0$ uniformly when
$|z| \to R$, by properness. Therefore $\omega (z) \to 0$ when
$|z| \to R$.
Consequently,  we may assume that
$\phi$ does not vanish on $U_R$.

Since the metric $ds^2$ is complete and $\omega (z) \to 0$
uniformly when $|z| \to R$, the new metric given by $ |\phi |\,|dz|^2$ is
complete too.

Since  $\phi$ is
 holomorphic, a result of Osserman
shows that $R=\infty$,
\cite[Lemma 9.3]{Osserman}. Thus, $ |\phi |\, |dz|^2$ is a complete metric on
$U $. Furthermore, using the fact that $\phi$ does not vanish,
another result of Osserman shows that $\phi$ has at most a pole at
infinity, \cite[Lemma 9.6]{Osserman}.

At last, recall that  $\phi$ is the square of a holomorphic function :
$\phi(z)=-(h_z (z))^2$, \cite[Proposition 1]{ST}.
This shows that $\phi$ has the required form. \qed

\bigskip

From now on we assume that $\phi$ has no zero on $U $.

\bigskip

\step \label{finite curvature}{\em The end $E$ has finite total curvature.}

\medskip

We first show the following result.

\begin{lemma}\label{L.nonzero}
  The polynomial function $P$ is not identically zero.
\end{lemma}

\begin{proof}
We proceed as in the proof of \cite[Lemma 2.1]{HNST}.
 Assume by contradiction that $P\equiv 0$.

 If $a_{-1}=0$ we have
\begin{equation*}
 \int_{U } \arrowvert \phi (z)\arrowvert \, dA < \infty.
\end{equation*}
where $dA$ is the Lebesgue measure on $\R^2$.
Since $\omega (z) \to 0$ when $\arrowvert z\arrowvert \to \infty$, this would
imply that the end $E$ has finite area, which is absurd, see
\cite[Appendix: Theorem 3 and Remark 4]{Frensel}.

If $a_{-1}\not=0$ the argument is the same as in the proof of
\cite[Lemma 2.1]{HNST}.
\end{proof}

\noindent {\em Proof of Step \ref{finite curvature}. }
We set $\sqrt{\phi (z)}=\sum_{k\geq 1}\frac{a_{-k}}{z^k} +P(z)$ and $m=\deg(P)$.
We define on $U  $ the, eventually,  multivalued function:
\begin{equation*}
 W(z):=\int \big(\sum_{k\geq 1}\frac{a_{-k}}{z^k} +P(z)\big)\, dz
 =\int \sqrt{\phi (z)}\, dz.
\end{equation*}
If $a_{-1}\not=0,$ the function $W(z)$ may be multivalued, but $\Im W(z)$ is
simply valued,
since $h(z)=2 \,\Im W(z)$. Noticing that  $W^\prime =\sqrt{\phi}$, the
holomorphic (possibly multivalued)
function $W$ has no critical
point.

\medskip

In \cite[Section 2]{HNST} its is showed that there exist
connected
and simply connected domains $\Omega_0,\dots,\Omega_{2m+1}$ in $\C$ such that:
\begin{itemize}
 \item $\{z\in \C \mid \arrowvert z\arrowvert >R\} \subset \cup_{k=0}^{2m+1}
 \Omega_k$ for some $R>1$.

 \item the restricted
 map  $W_k:=W_{| \Omega_k} : \Omega_k \rightarrow \C$ is an univalent map for
any
$k$.
\end{itemize}

For  $z\in \Omega_k$, $k=0,\dots,2m+1$, we set $w_k :=W(z)$. By abuse of
notation we just write $w=W(z)$. The range $\wt \Omega_k :=
W_k(\Omega_k)$ is a simply connected domain in $\C$ satisfying:
\begin{enumerate}
 \item If $k$ is an even number, then $\wt \Omega_k$ is the complementary of a
horizontal  half-strip. The non horizontal component of
$\partial \wt \Omega_k$ is a compact arc and $\Im\, w$ is strictly
monotonous along this arc. Moreover $\Re \, w$ is bounded  from above
by a real number $a_k$ along  $\partial \wt \Omega_k,$ see \cite[Figure 3 (a)]{HNST}.

\medskip

\item  If $k$ is an odd number, then $\wt \Omega_k$ is the complementary of
a horizontal  half-strip. The non horizontal component of
$\partial \wt \Omega_k$ is a compact arc and $\Im\, w$ is strictly
monotonous along this arc. Moreover $\Re \, w$ is bounded  from below
by a real number $b_k$ along $\partial \wt \Omega_k$, see \cite[Figure 3 (b)]{HNST}.
\end{enumerate}
Since $dw =\arrowvert \sqrt{\phi (z)} \arrowvert\, \arrowvert dz\arrowvert$ on
each $\Omega_k$,  we deduce from Formula (\ref{metric}) that the  metric induced
by the immersion $X$ on each $\wt \Omega_k$ is:
\begin{equation*}
 d\wt s^2 = 4\ch^2 \wt \omega (w)\, \arrowvert dw\arrowvert^2,
\end{equation*}
where, by abuse of notation, we set $\wt \omega (w)= (\omega\circ W_k^{-1})
(w)$.

Let $\Gamma\subset U  $ be a smooth Jordan curve non
homologous to zero
and let $C>>0$ be a large number satisfying
$\arrowvert W(z) \arrowvert << C$ for any $z\in \Gamma$.

Using \cite[Lemma 2.3]{HNST} and
the
description of each domain $\wt \Omega_k$, we can
construct smooth compact and simple arcs $\wt R_k(C), \wt I_k(C)$,
$k=0,\dots,2m+1$ such that:
\begin{itemize}
\item $\wt R_k(C)\subset \wt \Omega_k$ and
$\wt I_k(C) \subset \wt \Omega_k $.

\item If $k$ is an even number then: $\Re\, w =C$ and
$\Im\, w$ is increasing from $-C$ up to $C$
along $\wt R_k(C)$,
$\Im\, w =C$ and
$\Re\, w$ is increasing from $-C$ up to $C$
along $\wt I_k(C)$.
Moreover, the arcs $\wt R_k(C)$ and $\wt I_k(C)$ make a right angle at the
point $w=C+iC$.

\item If $k$ is an odd number then: $\Re\, w = -C$ and
$\Im\, w$ is increasing from $-C$ up to $C$
along $\wt R_k(C)$,
$\Im\, w = -C$ and
$\Re\, w$ is increasing from $-C$ up to $C$
along $\wt I_k(C)$.
Moreover, the arcs $\wt R_k(C)$ and $\wt I_k(C)$ make a right angle at the
point $w=-C -iC$.

\item Setting $R_k(C):= W_k^{-1}\big(\wt R_k(C)\big)$ and
$I_k(C):= W_k^{-1}\big( \wt I_k(C)\big)$, the
curve
$\Gamma (C) := \bigcup_{k=0}^{2m+1} \Big( R_k(C)\cup I_k(C)\Big)$ is a
piecewise smooth Jordan
curve non homologous to zero, and any of the
 $4m+4$ interior angles
is equal to $\pi/2$.
\end{itemize}

\medskip

Since $\arrowvert W(z) \arrowvert << C$ along $\Gamma$, we have
$\Gamma \cap \Gamma (C)=\emptyset$. We denote by $U(C)$ the annulus in
$U  $  bounded by $\Gamma$ and $\Gamma (C)$. To prove that
the end $E$ has
finite total curvature it suffices to show that $\int_{U(C)} K\, dA$ has finite
limit as $C\to +\infty$.

Since the boundary component $\Gamma$ is smooth and the other boundary
component  $\Gamma (C)$ has exactly $4m+4$ interior angles,
each one being
equal to  $\pi/2$, the Gauss-Bonnet formula gives
\begin{equation*}
 \int_{U(C)} K\, dA + \int_\Gamma k_g\,ds + \int_{\Gamma(C)} k_g\,ds
 =-2(m+1)\pi,
\end{equation*}
where $k_g$ denotes the geodesic curvature.
Therefore it suffices to show that $\int_{\Gamma(C)} k_g\,ds\to 0$ when
$C\to +\infty$.

First we prove that $\int_{I_k(C)} k_g\,ds\to 0$ when
$C\to +\infty$, $k=0,\dots,2m+1$. Since a similar argument shows also that
$\int_{R_k(C)} k_g\,ds\to 0$, we will be done.

Since $W_k : (\Omega_k, ds^2) \rightarrow (\wt \Omega_k, d\tilde s^2)$ is an
isometry, we have
\begin{equation*}
 \int_{I_k(C)} k_g\,ds = \int_{ \wt I_k(C)}k_g\,d\tilde s.
\end{equation*}
Assume that $k$ is even. We set $w=u+iv$ and we consider the parametrization
of $\wt I_k(C)$ given by $w(t)=t+iC$, $t\in [-C,C]$. Using
\cite[Formula (42.8)]{Kreyszig} we derive the geodesic
curvature:
\begin{equation*}
 k_g (w(t))= \pm\frac{\sh \wt \omega }{2\ch^2 \wt \omega}\,
 \frac{\partial \wt \omega}{\partial v } \big(w(t)\big).
\end{equation*}
Therefore,
\begin{equation*}
 \arrowvert k_g (w(t)) \arrowvert \leq
 \frac{1}{2\ch \wt \omega (w(t))}\,
 \arrowvert \grad \wt \omega (w(t))\arrowvert
\end{equation*}
where $\grad $ means the Euclidean gradient.
It is showed in the proof of \cite[Proposition 2.3]{HNST} that there exists a
positive constant $\delta$ such that, outside a compact part of
$\wt \Omega_k$, we have:
\begin{equation*}
\arrowvert \grad \wt \omega (w)\arrowvert <
\delta e^{-d(w,\partial \wt\Omega_k)},
\end{equation*}
where $d(w,\partial \wt\Omega_k)$ is the Euclidean distance between $w$ and
$\partial \wt\Omega_k$. Since, by construction, we have
$\arrowvert \Im (w)\arrowvert \leq C_0$ along $\partial \wt\Omega_k$
for some $C_0 < C$, we get
\begin{equation*}
 \arrowvert \grad \wt \omega (w)\arrowvert < \delta e^{C_0}\,e^{-C},\quad \text{on}\quad \wt I_k (C).
\end{equation*}
Therefore,
\begin{align*}
 \big\arrowvert  \int_{I_k(C)} k_g\,ds \big\arrowvert & \leq
 \int_{ \wt I_k(C)} \arrowvert k_g (w) \arrowvert \,d\tilde s \\
 &= \int_{-C}^C \arrowvert k_g (w(t)) \arrowvert\,
 2\ch \wt\omega (w(t))\, dt\\
 &\leq 2C  \delta e^{C_0}\,e^{-C}.
\end{align*}
We deduce that $\int_{I_k(C)} k_g\,ds\to 0$ when
$C\to +\infty$.

If $k$ is an odd number the argument is the same.

Along the curves $\wt R_k$, the geodesic curvature is given by
$\displaystyle k_g (w(t))= \pm\frac{\sh \wt \omega }{2\ch^2 \wt \omega}\,
 \frac{\partial \wt \omega}{\partial u } \big(w(t)\big)$.

 Moreover we have along $\partial \wt\Omega_k$: $\Re\, w \leq a_k$ if $k$ is an
even number and $\Re\, w \geq b_k$ if $k$ is an odd number.

 Therefore we can
proceed as before to show that $\int_{R_k(C)} k_g\,ds\to 0$ when
$C\to +\infty$.

 This proves that the end $E$ has finite total curvature. \qed

\bigskip

\step \label{converges}
{\em We have $\pain E =\pain P$.
Thus combining with Step \ref{asymptotic}, $E$ converges
to the vertical plane $P$.}

\medskip

\noindent {\em Proof of Step \ref{converges}.}
We  keep the notations of the proof of Step \ref{finite curvature}.

For $k=0,\dots,2m+1$, the map
$F_k:= F\circ W_k^{-1} : \wt \Omega_k \rightarrow \hi2$ is a harmonic map.
Assume that $k$ is an even number. It is proved in \cite[Theorem 2.1]{HNST} that
$\lim_{u\to +\infty} F_k( u+iC)$ exists and does not depend on $C\in \r$.
Also, $\lim_{u\to -\infty} F_{k+1}( u+iC)$ exists and does not depend
on $C\in \r$, moreover the two limits are different.

Since $\pain F(U  )=\{p_\infty,q_\infty \}$, we can assume that
$\lim_{u\to +\infty} F_k( u+iC)=p_\infty$, and
$\lim_{u\to -\infty} F_{k+1}( u+iC)=q_\infty$
for any $C$.

Let $t_0$ be any real number. We want to prove that
$(p_\infty, t_0) \in \pain E$. Since this will be true for any
$t_0\in \r$ and since we can show the same property for $q_\infty$, we could
conclude, using Step \ref{asymptotic} that $\pain E=\pain P$.

Let $C>0$ be large enough so that $C>> \arrowvert t_0\arrowvert$.
Let $\wt R_k(C)\subset \wt \Omega_k$ be the compact arc as in the proof of
Step \ref{finite curvature}. If $C$ is large enough then
$(X\circ W_k^{-1})(\wt R_k(C))$ is a compact arc of $E$ very close of
$p_\infty \times [-2C,2C]$ in the Euclidean meaning. Moreover the height is
increasing from $-2C$ to $2C$ along this arc
Letting $C\to +\infty$
we can extract a sequence $(p_n)$ on $E\cap \{t=t_0\}$ such that
$p_n \to (p_\infty, t_0)$. We have  therefore $\pain P\cap (\pain\hi2\times
\r)\subset \pain E $.
Taking into account  Step \ref{asymptotic} we conclude that $\pain P=\pain E$
and the end $E$ converges
to the vertical plane $P$.
\qed

\bigskip

\step \label{horizontal graph} {\em Assume that the end  $E$ is embedded. Then,
up to a compact part, the
end $E$ is a horizontal
graph.}

\medskip

\noindent {\em Proof of Step \ref{horizontal graph}.}

The end $E$ is conformally
parametrized  by
$U   :=\{ z \in \C \mid 1\leq |z| \}$. The conformal
immersion
$X:U   \rightarrow E\subset  \hi2 \times \r$ is given by $X=(F,h)$ where
$F: U  \rightarrow \hi2$ is a harmonic map and $h: U  \rightarrow \r$
is a harmonic function. Thus, $\phi= -h_z^2$ is holomorphic.
Moreover $\phi$ has the following form
\begin{equation}
 \phi (z)=(\sum_{k\geq 1}\frac{a_{-k}}{z^k} +P(z))^2,
\end{equation}
where $P$ is a polynomial function. We know from Lemma \ref{L.nonzero} that $P$
is not
identically zero. Let  $m\in \N$ be the degree of $P$. Since $X$ is an
embedding and since $\pain F(U)=\{p_\infty,q_\infty\}$, we deduce from
\cite[Theorem 2.1]{HNST} that $m=0$.

Let $\varepsilon >0$ be a small number, $\varepsilon <<1$.
Since $\phi$ does not vanish on $U$, we have
$|n_3|\not=1$ throughout $E$, and therefore
the end
$E$ is transversal to any
slice $\hi2 \times \{t\}$. For any $p\in E$ we denote by $\kappa (p)$ the
geodesic curvature, in $\hi2$, of the intersection curve between $E$ and the
horizontal slice through $p$. We deduce from \cite[Proposition 2.3]{HNST} that,
there exists a compact part $K$ of $E$ such that
\begin{equation}\label{F.estimate-geodesic}
 \arrowvert\kappa (p)\arrowvert < \varepsilon,
\end{equation}
for any $p\in E\setminus K$.

\smallskip

Let $\gamma_1 \subset \hi2$ be a geodesic orthogonal to $\gamma$, thus
$\pain \gamma \cap \pain \gamma_1 =\emptyset$.  Let $\rho>0$, we denote
by $L_\rho^+$ and $L_\rho^-$ the two equidistant curves of $\gamma_1$ at
distance $\rho$. Let
$Z_\rho^\prime $ be the connected component of
$(\hi2\times\r) \setminus ( L_\rho^+ \cup L_\rho^-)\times \r$ containing
the vertical plane $P_{\gamma_1}$.

We deduce from \cite[Proposition 4.3]{HNST} and (\ref{F.estimate-geodesic})
that for $\rho$ large enough
$E\setminus Z_\rho^\prime$ is a horizontal graph with respect to $\gamma_1$.

We infer from \cite[Theorem 2.1]{HNST} that there exists $C_0>0$ such that for
any $t$ satisfying $\arrowvert t \arrowvert >C_0$, the intersection
$E\cap \hi2 \times \{t\}$ is a complete and connected curve $L_t$ which is
$C^1$-close to $\gamma$. Therefore, $L_t \cap Z_\rho^\prime$ is
a horizontal graph with respect to $\gamma_1$.

Thus, up to a compact part, $E$ is a horizontal graph. This achieves the proof
of Step \ref{horizontal graph} and concludes the proof of
Theorem \ref{Main Theorem}. \qed

\bigskip

\section{Appendix.}

\begin{prop}\label{comparaison}
 Let $U\subset \r^3$ be an open set and let $S\subset U$ be an immersed
 $C^2$-surface without boundary. Let $g$ be a $C^1$-metric on $U$, and denote
by $g_{euc}$ the Euclidean metric.

Let $A$ and $\ov A$ be the second fundamental forms of $S$ for, respectively,
the metrics $g_{euc}$ and $g$.

Assume there exist  positive constants $C_1, C_2$ such that
\begin{itemize}
\item $\arrowvert \ov A \arrowvert <C_1$ on $S$
\item $\arrowvert g_{ij} -g_{euc,ij}\arrowvert_{C^1(U)}< C_2$ and
$\arrowvert g^{ij} -g_{euc,ij}\arrowvert_{C^0(U)}< C_2$,
$1\leq i,j \leq 3$.
\end{itemize}
Then, there is a constant $C_3 >0$, depending on $C_1$ and $C_2$ and not on
$S$, such that  $\arrowvert A \arrowvert < C_3$ on $S$.
\end{prop}

\begin{proof}
 Let $p\in S$ and  let $v\in T_p S$ be a non zero tangent vector. We choose
Euclidean  coordinates $(x_1,x_2,x_3)$ on $U$ so that $p=(0,0,0)$, the tangent
plane
$T_pS$ coincides with the plane $\{x_3=0\}$ and $v$ is tangent to the
$x_1$-axis.

With these new coordinates we certainly have
$\arrowvert g_{ij} -g_{euc,ij}\arrowvert_{C^1(U)}< 9C_2$
and also \linebreak
$\arrowvert g^{ij} -g_{euc,ij}\arrowvert_{C^0(U)}< 9C_2$,
$1\leq i,j \leq 3$.

Thus, a part of $S$ is the graph of a function $u$ defined in a neighborhood of
the origin in the plane  $\{x_3=0\}$.

Let $\lambda (v)$, resp. $\ov \lambda (v)$, be the normal curvature of $S$ at
$p$ in the direction $v$ for the metric $g_{euc}$, resp. $g$. Both curvatures
being computed with respect to normals inducing the same transversal
orientation along $S$.

A straightforward computation shows that
\begin{align*}
 \lambda (v) &= u_{11}(0) \\
 \ov \lambda (v) &= \frac{1} {g_{11}\sqrt{g^{33 }}} \Big[
 g\Big(\ov \nabla_{\partial_{x_1}}\partial_{x_1}, g^{13}\partial_{x_1}
 + g^{23}\partial_{x_2} + g^{33}\partial_{x_3}\Big) (0) + u_{11}(0)\Big],
\end{align*}
where $\ov \nabla$ denotes the Riemannian connection of $(U,g)$.
Therefore,
\begin{equation*}
 \lambda (v)= \big(g_{11}\sqrt{g^{33 }}\big)(0) \, \ov \lambda ( v) -
 g\Big(\ov \nabla_{\partial_{x_1}}\partial_{x_1}, g^{13}\partial_{x_1}
 + g^{23}\partial_{x_2} + g^{33}\partial_{x_3}\Big) (0).
\end{equation*}
Since $\arrowvert g_{ij} -g_{euc,ij}\arrowvert_{C^1(U)}< 9 C_2$ and
$\arrowvert g^{ij} -g_{euc,ij}\arrowvert_{C^0(U)}<9 C_2$,
$1\leq i,j \leq 3$, there is a constant $M>0$, depending only on $C_2$ and not
on $S$, such that
\begin{equation*}
 \arrowvert g\Big(\ov \nabla_{\partial_{x_1}}\partial_{x_1},
g^{13}\partial_{x_1}
 + g^{23}\partial_{x_2} + g^{33}\partial_{x_3}\Big) (0)\arrowvert <M \quad
\text{and} \quad
 g_{11}\sqrt{g^{33 }}(0) <M.
\end{equation*}
Therefore we obtain
\begin{equation*}
 \arrowvert \lambda (v)\arrowvert \leq M \arrowvert \ov \lambda (v)\arrowvert
+M \leq M(\arrowvert \ov A (p)\arrowvert+1) < M(C_1+1),
\end{equation*}
so that it suffices to choose $C_3=2 M(C_1+1)$.
\end{proof}

\end{document}